\documentclass[12pt]{amsart}
\usepackage{amsmath}
\usepackage[scaled]{helvet}
\usepackage[T1]{fontenc}
\usepackage{textcomp}
\normalfont
\usepackage{amssymb}
\usepackage[usenames,dvipsnames]{color}

\usepackage{graphicx}
\usepackage[utf8]{inputenc}
\usepackage[english]{babel}
\usepackage{csquotes}
\usepackage{amssymb}
\usepackage{tikz-cd}

\usepackage{fullpage}
\usepackage{amsmath, amsfonts,hyperref,verbatim,comment}
\usepackage{amssymb}

\usepackage[normalem]{ulem}

\usepackage{todonotes, pictex}

\usepackage{enumerate}
\usepackage{enumitem}

\newtheorem{theorem}{Theorem}[section]
\newtheorem{lemma}[theorem]{Lemma}
\newtheorem{proposition}[theorem]{Proposition}
\newtheorem{corollary}[theorem]{Corollary}
\newtheorem{conjecture}[theorem]{Conjecture}

\newtheorem{definition}[theorem]{Definition}

\theoremstyle{remark}

\numberwithin{equation}{section}

\newcommand{\mc}{\mathcal}

\newcommand{\la}{\lambda}

\newcommand{\norm}[1]{\left\lVert #1 \right\rVert}
\newcommand{\inner}[1]{\left\langle #1 \right\rangle}

\newcommand{\R}{\mathbb{R}}

\DeclareMathOperator{\proj}{proj}

\newcommand{\holder}{H\"{o}lder }
\newcommand{\til}{\widetilde}

\newcommand{\cout}[1]{}
\DeclareMathOperator{\dvol}{\mu}
\definecolor{darkcyan}{rgb}{0. 0.65, 0.65}

\newcommand{\hrk}{rk^{h}}
\newcommand{\shrk}{rk^{sh}}

\newtheorem{Structural Stability Theorem}[theorem]{Structural Stability Theorem}

\def\bt{\begin{theorem}}
\def\et{\end{theorem}}
\def\bd{\begin{definition}}
\def\ed{\end{definition}}
\def\bl{\begin{lemma}}
\def\el{\end{lemma}}

\def\be#1\ee{\begin{align}\begin{split} #1 \end{split}\end{align}}
\def\beq#1\eeq{\begin{align*}\begin{split} #1 \end{split}\end{align*}}

\begin{document}

\title{Hyperbolic Rank Rigidity for Manifolds of   $\frac14$-Pinched Negative Curvature}

\author{Chris Connell$^\dagger$, Thang Nguyen, Ralf Spatzier$^{\ddagger }$}

\address{Department of Mathematics,
Indiana University, Bloomington, IN 47405}
\email{connell@indiana.edu}

\address{Courant Institute  of Mathematical Sciences, New York University, New York, NY 10012}
\email{tnguyen@nyu.edu}

\address{Department of Mathematics, University of Michigan, 
    Ann Arbor, MI, 48109.}
\email{spatzier@umich.edu}

\thanks{$^\dagger$ Supported in part by Simons Foundation grant \#210442}
\thanks{$^{\ddagger}$ Supported in part by NSF grants DMS 1307164 and DMS 1607260}


\date{}

\begin{abstract}
A Riemannian manifold $M$ has  higher  hyperbolic rank if every geodesic has a
perpendicular  Jacobi field  making sectional curvature -1  with the geodesic.
If in addition, the sectional curvatures of $M$  lie in the interval
$[-1,-\frac14]$, and $M$ is closed, we show that $M$ is a locally symmetric
space of rank one. This partially extends  work by Constantine using completely
different methods. It  is also a partial converse to Hamenst\"{a}dt's hyperbolic
rank rigidity result for sectional curvatures $\leq -1$, and complements
well-known results  on Euclidean and spherical rank rigidity.
  \end{abstract}

\maketitle

\section{Introduction}

 Given a closed Riemannian manifold $M$ and  a unit vector $v\in SM$, we define the {\em  hyperbolic rank }  $\hrk (v)$ of $v$ as the dimension of the subspace of $v^\perp\subset TM$ which are the initial vectors of a Jacobi field $J(t)$ along $g_tv$ which spans a plane of sectional curvature $-1$ with $g_tv$ for all $t\ge 0$ (where $J(t) \neq 0$).  
The {\em  hyperbolic rank } of $M$, $\hrk (M)$, then is the infimum of   $\hrk (v)$ over all  unit vectors $v$.  We also say that $M$ has {\em  higher hyperbolic rank} if $\hrk (M) >0$.  

Our notion of  hyperbolic rank is   a priori weaker than either the usual  one which requires that the Jacobi fields in question make curvature $-1$ for $t\in (- \infty, \infty)$ or else the version that uses parallel fields in place of Jacobi fields.  In strict negative curvature these distinct formulations turn out to coincide (see Corollary \ref{strong=weak}).
Actually, the techniques of our proofs require us to introduce the notion of hyperbolic rank for positive time. 

 The main goal of this paper is the following hyperbolic rank rigidity result.

\bt  \label{thm:main}
Let $M$ be  a closed Riemannian manifold of higher hyperbolic rank and sectional curvatures $K$ between $-1  \leq K \leq-\frac14$.  Then $M$ is a rank one locally symmetric space. In particular, if the pinching is strict then $M$ has constant curvature $-1$.
\et

Constantine \cite[Corollary 1]{Constantine-rank} characterized constant curvature manifolds among those of nonpositive curvature under one of two conditions: odd dimension without further curvature restrictions, or even dimension provided the sectional curvatures  are pinched between $-(.93)^2$ and $-1$. He also showed that if one uses the stronger notion of parallel fields in place of Jacobi fields then one may relax the lower curvature bound of $-1$, though still requiring the same pinching in even dimensions.  His method  is rather different from ours, drawing on ergodicity  results for the  2-frame flow of  such manifolds.
For   $\frac 14$-pinched manifolds of negative curvature however, ergodicity of the frame flow  has been   conjectured now for over 30 years, with no avenue for an approach in sight \cite[Conjecture 2.6]{Brin82}. 

 Both Constantine's and our result are  counterpoints to Hamenst\"{a}dt's hyperbolic rank rigidity theorem     \cite{Hamenstadt91}:   

\bt {\em (Hamenst\"{a}dt)} \label{thm:Hamenstadt}
Closed manifolds with   sectional curvatures $K \leq -1$   and higher hyperbolic rank are locally symmetric spaces of real rank 1.  
\et

 Compactness is truly essential in these results. Indeed, Connell found a counterexample amongst homogeneous manifolds of negative curvature whilst   proving hyperbolic rank rigidity for such spaces   under an additional condition \cite{Connell2}.
  
Lin and Schmidt recently  constructed non-compact manifolds of higher hyperbolic rank in \cite{Lin16} with both upper and lower curvature bounds $-1$ and curvatures  arbitrarily pinched.  In addition, their examples are not even locally homogeneous and every geodesic lies in a totally geodesic hyperbolic plane.
 In dimension three, Lin  showed that finite volume  manifolds with higher hyperbolic rank always have  constant curvature, without imposing any curvature properties \cite{Lin16}.

The notion of hyperbolic rank  is analogous to that of {\em strong Euclidean and spherical rank} where we are looking for parallel vector fields (not just Jacobi fields) along geodesics that make curvature 0 or 1 respectively.   When 0, 1 or -1 are also extremal as values of sectional curvature, various rigidity theorems have been proved.  In particular we have the results of Ballmann and Burns-Spatzier  in nonpositive curvature where  higher rank Euclidean manifolds are locally either  products or symmetric spaces
 (cf. \cite{Ballmann, BallmannBook, BS},  Eberlein and Heber \cite{ EH} for certain noncompact manifolds and Watkins \cite{Watkins} for no-focal points).
  When the sectional curvatures   are less than 1, and $M$ has  higher spherical rank,  Shankar, Spatzier and Wilking showed that $M$ is locally isometric to a compact rank one symmetric space  \cite{SSW}.   Notably, there are counterexamples in the form of the Berger metrics for the analogous statements replacing Jacobi fields by parallel fields in the definition of higher spherical rank (see \cite{SSW}).

  Thus the situation for closed manifolds is completely understood for upper curvature bounds, and we have full rigidity.   For lower curvature bounds, the situation is more complicated.  For one, there are  many closed manifolds of nonnegative curvature and higher Euclidean rank.  The first examples were given  by Heintze (private communication) and were still homogeneous.  More general and in particular inhomogeneous  examples were constructed by Spatzier and Strake in \cite{SpatzStrake}.  For higher spherical rank and lower bound on the sectional curvature by 1, Schmidt, Shankar and Spatzier again proved local isometry to a sphere of curvature 1 if the  spherical rank is at least $n -2 >0$, $n$ is odd or if $n \neq 2, 6$ and $M$ is a sphere \cite{SchShSP2016}.  No counterexamples are known.  If $M$ in addition is K\"{a}hler of dimension at least 4, then $M$ is locally isometric to complex projective space with the Fubini-Study metric.  In dimension 3, Bettiol and Schmidt showed that higher rank implies local splitting of the metric, without any conditions  on the curvature  \cite{Bettiol16}.

Let us outline our argument for Theorem \ref{thm:main} which occupies the remainder of this paper. In fact, all of our arguments hold for manifolds of with sectional curvature bounds $-1\leq K <0$ until Section \ref{sec:kanai}.  We show that we may   assume that every geodesic $c(t)$ has orthogonal parallel fields $E$ with sectional curvature $-1$.    The dimension of the latter vector space is called the  {\em strong hyperbolic rank} of $c$.  
Following Constantine in \cite[Section 5]{Constantine-rank}, strong rank agrees with the rank under lower sectional curvature bound $-1$ (cf. Proposition  \ref{prop:weakstrong}).
Then we show in Section \ref{sec:regular} that the regular set $\mc{R}$ of unit tangent vectors $v$ for which $\hrk (v) = \hrk (M)$ is dense  and open. Additionally it has the property that if $v \in \mc{R}$ is recurrent then its stable and unstable manifolds also belong to $\mc{R}$. 
Next in Section \ref{sec:smooth}, we show that the distribution of parallel fields of curvature $-1$ is smooth on the regular set.  Then, for bi-recurrent regular vectors,  we characterize these parallel fields in Section \ref{sec:lyapunov} in terms of unstable Jacobi fields of Lyapunov exponent 1. We use this to show that the slow unstable distribution extends to a smooth distribution on $\mc{R}$.

In Section \ref{sec:kanai}, we prove the result under the stronger assumption of strict $\frac14$-curvature pinching as the technicalities are significantly simpler and avoid the use of measurable normal forms from Pesin theory. We are inspired here by arguments of Butler in \cite{Butler15}.
  We construct a Kanai like connection for which the slow and fast stable and unstable distributions are parallel.  The construction is much motivated by a similar one by Benoist, Foulon and Labourie in \cite{BFL90}.  We use this to prove integrability of the slow unstable distribution.   This distribution is also invariant under stable holonomy by an argument of Feres and Katok \cite{FeresKatok1990}, and hence defines a distribution on $\partial \til{M}$.  As it is integrable and $\pi_1(M)$-invariant, we get a $\pi_1(M)$-invariant foliation on  $\partial \til{M}$ which is impossible thanks to an argument of Foulon \cite{Foulon1994} (or the argument for   Corollary 4.4 in \cite{Hamenstadt91}.) 
  
Lastly, in Section \ref{sec:non-strict} we treat the general case of non-strict $\frac14$-curvature pinching. By a result of Connell \cite{Connell03} relying on Theorem \ref{thm:Hamenstadt}, if $M$ is not already a locally symmetric space, then there is no uniform $2:1$ resonance in the Lyapunov spectrum. Now we can use  recent work of Melnick \cite{Melnick16} on normal forms to obtain a suitably invariant connection (cf. also Kalinin-Sadovskaya \cite{Kalinin13}). This allows us to prove integrability of the slow unstable distribution on almost every unstable manifold. As before we can obtain a $\pi_1$-invariant foliation on $\partial\til{M}$ and finish with the result of Foulon as before. This is technically more complicated,  however, because we no longer have $C^1$ holonomy maps. Instead we adapt an argument of Feres and Katok, to show that stable holonomy maps almost everywhere preserve the tangencies of our slow unstable foliation.  To this end, we show that the holonomy maps are differentiable with bounded derivatives, though not necessarily $C^1$, between good unstable manifolds. This allows us to obtain the desired holonomy invariance as in the strict $\frac14$-pinching case to finish the proof of the main theorem.

In light of the above, in particular  Theorem \ref{thm:main} as well as Constantine's results, we make the following
   
 \begin{conjecture} 
 A closed manifold with sectional curvatures $\geq -1$ and higher  hyperbolic rank is isometric to a locally symmetric space of real rank 1. 
 \end{conjecture}

  Let us point out that the starting point of the proofs for upper and lower curvature bounds are radically different, although they share some common features.  
 In the hyperbolic rank case in particular, for the upper curvature bound, we get control of the slow unstable foliation in terms of parallel fields.  Hamenst\"{a}dt used the latter to create  Carnot metrics on the boundary with large conformal group leading to the models of the various hyperbolic spaces.  The lower curvature bound in comparison gives us control of the fast unstable distribution which is integrable and does not apparently tell us anything about the slow directions.  It is clear that the general case will be much more difficult, even if we assume that the metric has negative or at least non-positive curvature. 
    \vspace{1em}
    
     Finally let us note a consequence of Theorem \ref{thm:main} in terms of dynamics.    Consider the geodesic flow $g_t$ on the unit tangent bundle of a closed   manifold $M$.  For a geodesic $c \subset M$, the {\em maximal Lyapunov exponent} $\lambda  _{max} (c)$,   for $c$ is the biggest exponential growth rate of the norm of a Jacobi field $J(t)$ along $c$:
    $$  \lambda  _{max} (c) := \max _{J  \text{ Jacobi for } c} \lim \frac 1t \log \|J(t)\| .$$
Note that $ \lambda  _{max} (c) \leq 1$ if the sectional curvatures of $M$ are bounded below by $-1$,    by Rauch's comparison theorem.

Given   an ergodic $g_t$-invariant measure $\mu$ on the unit tangent bundle $SM$, $\lambda  _{max} (c)$ is constant $\mu$-a.e..  In fact, it is just the maximal Lyapunov exponent in the sense of dynamical systems for  $g_t$ and $\mu$ (cf. Section \ref{sec:lyapunov}).

\begin{corollary}\label{cor:measures}
Let $M$ be  a closed Riemannian manifold with sectional curvatures $K$ between $-1 \leq K \leq -\frac14$.  Let $\mu$ be a probability measure  of full support on the unit tangent bundle $SM$ which is invariant and ergodic under the geodesic flow $g_t$.  Suppose that the maximal Lyapunov exponent for $g_t$and $\mu$ is $1$.   Then $M$ is a rank one locally symmetric space.
\end{corollary} 

We supply a proof in Section \ref{sec:non-strict}. In fact, the reduction to Theorem \ref{thm:main} is identical to  Constantine's  in \cite[Section 6]{Constantine-rank} which in turn adapts an argument of Connell for upper curvature bounds \cite{Connell03}.

\vspace{1em}  

\noindent {\em Acknowledgements:}   The third author thanks the Department of Mathematics at Indiana University for their hospitality while part of this work was completed.

\section{Definitions, Semicontinuity and Invariance on Stable Manifolds}  \label{sec:regular}

Let M be compact manifold of negative sectional curvature, and denote its unit tangent bundle by $SM$.   We let $g_t:  SM \to SM$ be the geodesic flow, and  denote by $pt: SM \to M $ the footpoint map, i.e. $v \in T_{pt (v) } M$. 
For $v \in SM$, let $c_v$ be the geodesic determined by $v$ and let $v^{\perp}$ denote the perpendicular complement of $v$ in $T_{pt (v) } M$. Recall that $\hrk (v)$ is the dimension of the subspace of $v^\perp$ which are the initial vectors of Jacobi fields that make curvature $-1$ with $g_tv$ for all $t\ge 0$, and $\hrk (M) $ is the minimum of $\hrk (v)$ for $v \in SM$.

\begin{lemma} \label{lem:onesided}
Let $v$ be a unit vector recurrent under the geodesic flow. Suppose that $\hrk(v) >0$.  Then there is also an unstable or stable Jacobi field making curvature -1 with $g_tv$ for all $t \in \R$.
\end{lemma} 
\begin{proof}
Since $\hrk(v)>0$, there is a Jacobi field $J(t)$ making curvature -1 with $g_tv$ for all $t\ge  0$. First assume that $J(t)$ is not stable.
Decompose $J(t)$  into its stable and unstable components  $J(t) = J^s (t) + J^u (t)$.    Suppose $g_{t_n} v \rightarrow v$ with $t_n \rightarrow \infty$.  Then, for a suitable subsequence of $t_n$, 
$\frac{J(t + t_n)}{\| J(t_n) \|}$   will converge to a  Jacobi field $Y(t)$ along $c_v (t)$.   Note then  that $g_{t + t_n} (v)  \rightarrow g_t v$  as $t_n \rightarrow \infty$.  Moreover, for any $t \in \R$, $Y(t)$ is the limit of the vectors $J(t+ t_n)$ which make curvature -1 with $g_{t + t_n} (v)$.  Hence  $Y(t)$ also makes curvature -1 with $g_tv$ for any $t$.  Also $Y(t)$ is clearly unstable since $J^u(t)\not\equiv 0$.   

 If $J(t) = J^s (t)$ is stable, then the same procedure will produce  a stable Jacobi field $Y(t)$ along $c(t)$ that makes curvature $-1$ with $g_t (v)$ for all $t \in \R$. 
\end{proof}

\bl
Suppose that $\hrk (M) >0$.  Then along every geodesic $c(t)$, we have an unstable Jacobi field that makes curvature $-1$ with $c (t)$  for all $t \in \R$.  Similarly, there is a stable Jacobi field along $c(t)$  that makes curvature $-1$ with $c(t)$  for all $t \in \R$.
\el

\begin{proof}
Since the geodesic flow for $M$ preserves the Liouville measure $\dvol$, $\dvol $-a.e. unit tangent vector $v$ is recurrent.  By Lemma \ref{lem:onesided}, the geodesics $c_v (t)$ have  stable or unstable Jacobi fields along them that make curvature $-1$ with the geodesic for all $t \in \R$.  As $\dvol $ has full support in $SM$, such geodesics are dense and the same is true for  any geodesic by taking limits.

Next we show that there are both stable and unstable Jacobi fields along any geodesic  that make curvature $-1$ with the geodesic.  Indeed, let $A^+ \subset SM$ be the set of unit tangent vectors $v$ that have an unstable Jacobi field along $c_v (t)$ that make curvature $-1$ with $c_v (t)$. 
Similarly, define $A^- \subset SM$ as the set of unit tangent vectors $v$ that have a stable Jacobi field along $c_v (t)$ that make curvature $-1$ with $c_v (t)$. 
Note that $A ^- = - A^+$, and that $SM= A^+ \cup A^-$ by what we proved above.  Hence neither  $A^+$ nor $A^-$ can have  measure 0  w.r.t. Liouville measure $\dvol$. Also, both $A^+$ and $A^-$ are invariant under the geodesic flow $g_t$.  Since $g_t$ is ergodic w.r.t. $\dvol$, both  $A^+$ and $A^-$ must each have full measure.  Now the claim is clear once again by taking limits.
\end{proof} 

 Denote by $\Lambda(v,t)w$ the unstable Jacobi field along $g_tv$ with initial value $w\in v^\perp$. Then we let $\mc{E} (v) \subset v^{\perp}$ be the subspace of  $v^{\perp}$ defined as follows:  $w \in v^{\perp}$ belongs to $\mc{E} (v)$ if $\Lambda(v,t)w$ makes curvature -1 with $g_tv$ for all $t\ge 0$. 

 We define $\mc{R} = \{v \mid \hrk  (v) = \hrk  M\}$. We note that for $v\in \mc{R}$ and for all $u\in SM$, $\dim \mc{E}(v) \le \dim \mc{E}(u)$.

\bl   \label{lem:twosided}
Suppose $v \in \mc{R}$ and $ w \in \mc{E}  (v)$.  Then $\Lambda(v,t)w$  makes curvature  -1 with $c_v (t) $ for all $ t \in \R$ and $\mc{R}$ is invariant under the backward geodesic flow. 
\el

\begin{proof}
First note that for  $t\in \mathbb{R}$, the unstable Jacobi field $\Lambda(v,t): v^\perp\to (g_tv)^\perp$, defined by $w\mapsto \Lambda(v,t)w$ is an isomorphism. We have by definition that $\Lambda(v,t)\mc{E}(g_{-t}v)\subset \mc{E}(v)$ for $t>0$.  Since  $v\in \mc{R}$, we have $\dim \mc{E}(v) \le \dim \mc{E}(g_{-t}v)$. Thus  $\Lambda(v,t)\mc{E}(g_{-t}v)= \mc{E}(v)$ for $t>0$. Therefore, for $w\in \mc{E}(v)$, the Jacobi field $\Lambda(v,t)w$ along $c_v$ makes curvature -1 with $g_tv$ for all $t\in \mathbb R$. This immediately implies the last statement.
\end{proof}

 Next,   define $\widehat{\mc{E}} (v) \subset v^{\perp}$ be the subspace of  $v^{\perp}$ defined as follows:  $w \in v^{\perp}$ belongs to $\widehat{\mc{E}} (v)$ if  the parallel vector field along  $c_v (t)$  determined by w makes curvature -1 with $g_tv$ for all $t\in \R$. We have that  $\widehat{\mc{E}} (v)\subset \mc{E}(v)$.  Indeed  if $E (t)$ is a parallel vector field along a geodesic $c(t)$ that makes curvature -1 with $c(t)$, then $e^t E(t)$ is an unstable Jacobi field that again makes curvature -1 with $c(t)$.  

\begin{definition}  \label{def:stronghyperbolicrank}
The {\em strong hyperbolic rank} $\shrk (v)$ of $v$ is the dimension of $\widehat{\mc{E}} (v)$.  The {\em strong hyperbolic rank} $\shrk (M)$ of M is the minimum of  the strong hyperbolic ranks $\shrk (v)$ over all  $v\in SM$. 
\end{definition}

   We use an argument of  Constantine \cite[Section 5]{Constantine-rank} to prove:

\begin{proposition}  \label{prop:weakstrong}
If $M$ is a closed manifold with lower sectional curvature bound $-1$,   $v \in \mc{R}$ and $w \in \mc{E} (v)$,  then the parallel vector field determined by $w$ along $c_v (t)$ makes curvature $-1$ for all $t \in \R$. 
 Thus for all $v\in\mc{R}$, $\hrk(v)=\shrk(v)$ and $\widehat{\mc{E}} (v)={\mc{E}} (v)$ .
\end{proposition}

\begin{proof}  By Lemma \ref{lem:twosided}, the unstable Jacobi field  $\Lambda (v,t) w$ makes curvature $-1$ with $c_v (t)$ for all $t \in \R$.   Then $\Lambda (v,t) w$ is a stable Jacobi field along $c_{-v} (t)$ still making curvature -1 with $c_{-v} (t)$.  Hence the discussion in \cite[Section 5]{Constantine-rank} shows that $\Lambda (v,t) w = e^t E$ where $E$ is parallel along $c_v (t)$ for all $t \in \R$.  Clearly, $E$ makes sectional curvature -1 with $c_v (t)$ as well. 
\end{proof}

   Note that $\mc{E}$ and $\widehat{\mc{E}} $ may not be continuous a priori.  However, $\mc{E}$ and  $\widehat{\mc{E}} $ are  semicontinuous in the following sense. 

\begin{lemma}  \label{semicontinuity}
If $v_n, v \in SM$ and $v_n \rightarrow v$ as $n \rightarrow \infty$, then 
\begin{enumerate}
\item $\lim _{n \rightarrow \infty} \mc{E} (v_n) \subset \mc{E} (v)$ and $\hrk (v) \geq \limsup_{n\to\infty} \hrk (v_n)$
\item $\lim _{n \rightarrow \infty} \widehat{\mc{E}} (v_n) \subset \widehat{\mc{E}} (v)$ and $rh^{sh} (v ) \geq \limsup_{n\to\infty} \shrk (v_n)$.
\end{enumerate}
Here $\lim_{n\to\infty} \mc{E}  (v_n)$ simply denotes the set of all possibly limit points of vectors in $\mc{E} (v_n)$, and similarly for $\widehat{\mc{E}}$.
\end{lemma} 

\begin{proof} These claims are clear.  		\end{proof}

    We now define $\widehat{\mc{R}} = \{v \mid \shrk  (v) = \shrk  M\}$.

\begin{lemma}\label{open-dense-R}
 The sets $\mc{R}$  and $\widehat{\mc{R}}$ are both  open  with full measure and  hence dense. Moreover, $\widehat{\mc{R}}$ is invariant under the geodesic flow.
\end{lemma} 
\begin{proof}
By Lemma \ref{semicontinuity}, $\mc{R}$ is open.  Since the geodesic flow is ergodic on SM w.r.t. Liouville measure  and $\mc{R}$ is invariant under backward geodesic flow by Lemma \ref{lem:twosided}, $\mc{R}$ has full measure.  By Lemma \ref{semicontinuity}, $\widehat{\mc{R}}$ is open and it is flow invariant by definition. Therefore the same argument applies. \end{proof}

\begin{corollary}\label{strong=weak}
If $M$ is a closed manifold with lower sectional curvature bound $-1$, then $\hrk (M) = \shrk (M)$ and $\mc{R}\subset \widehat{\mc{R}}$.
\end{corollary}
\begin{proof}
By Lemma \ref{prop:weakstrong}, strong and weak rank agree on $\mc{R}$  which is an open dense set by Lemma  \ref{open-dense-R}.  By Lemma \ref{semicontinuity}, both weak and strong ranks can only go up outside $\mc{R}$. \end{proof}

   The next argument is well known and occurs in Constantine's work for example.  

   As usual we let $W^u (v)$ denote the (strong) unstable manifold of $v$ under the geodesic flow, i.e. the vectors $w \in SM$ such that $d(g_t (v), g_t (w) )\rightarrow 0$ as $ t \rightarrow  - \infty$. We define the (strong) stable manifold $W^s(v)$ similarly for $t \rightarrow \infty$.

\begin{lemma}\label{recurrent}
If $v\in \widehat{\mc{R}}$ is backward recurrent under $g_t$, then $W^u(v)\subset \widehat{\mc{R}}$. If $v\in \widehat{\mc{R}}$ is forward recurrent under $g_t$, then $W^s(v)\subset \widehat{\mc{R}}$.
\end{lemma}

\begin{proof} Let $w\in W^u(v)$, then $g_{-t}w$ approximates $g_{-t}v$ when $t$ large. On the other hand, since $\widehat{\mc{R}}$ is open, there is a neighborhood $U$ of $v$ in $\widehat{\mc{R}}$. Since $v$ is backward recurrent, $g_{-t}v$ comes back to $U$ and approximates $v$ infinitely often. Thus there is $t$ large that $g_{-t}w\in U\subset \widehat{\mc{R}}$. It follows that $w\in \widehat{\mc{R}}$ as $\widehat{\mc{R}}$ is invariant under the geodesic flow (cf. Lemma \ref{open-dense-R}). 
The argument for the forward recurrent case and stable leaf is similar.\end{proof}

\section{Smoothness of Hyperbolic Rank}   \label{sec:smooth}

   Assume now that M has sectional curvature -1 as an extremal value, that is, either the sectional curvature $K \leq -1$ or $K \geq -1$.  We want to prove smoothness of $\widehat{\mc{E}}$ on the  regular set $\widehat{\mc{R}}$. 
Our arguments below are inspired by Ballmann, Brin and Eberlein's work \cite{Ballmann} and also \cite{Watkins}.
First let us recall a lemma from \cite[Lemma 2.1]{SchShSP2016}:  

\begin{lemma}\label{Jacobi operator} 
For $v \in S_p M$,  the {\em Jacobi operator} $R_v : v^{\perp} \to v^{\perp}$ is defined by $R_v (w) = R(v,w)v$.  
Then w is an eigenvector of $R_v$ with eigenvalue -1  if and only if $K(v,w) = -1$.
\end{lemma}

   While we don't use it, let us mention \cite[Lemma 2.9]{SchShSP2016} where smoothness of the eigenspace distribution of eigenvalue -1 is proved on a similarly defined regular set.   Our situation is different as we characterize hyperbolic rank in terms of parallel transport of a vector not just the vector.  To this end, we define the following quadratic form:  Let $E(t)$ and $W(t) $ be  parallel fields along the geodesic $c_v (t)$, and set 

$$\Omega _v ^T  (E(t), W(t)) = \int _{-T} ^T \langle - E(t)  - R_{g_t v} E(t), - W(t)  - R_{g_t v} W(t) \rangle.$$

\begin{lemma} 
The parallel field $E(t)$ belongs to the kernel of $\Omega _v ^T$ if and only if $ E(t)$ makes curvature -1 with $c_v (t)$ for $t \in [-T,T]$.  In consequence, if $S<T$, then $\ker \Omega _v ^T \subset \ker \Omega _v ^S$.
\end{lemma}

\begin{proof} If $ E(t)$ makes curvature -1 with $c_v (t)$ for $t \in [-T,T]$, then $- E(t)  - R_{g_t v} E(t) = 0$ by Lemma \ref{Jacobi operator}, and hence $E(t)$ is in the kernel of $\Omega _v ^T$.  

    Conversely, if $E(t)$ is in the kernel of $\Omega _v ^T$, let $W(t)= E(t)$.  Since the integrand now is $\ge 0$ for all    $t \in [-T,T]$, $ E(t)  - R_{g_t v} E(t)=0$ and hence $E(t)$ makes curvature -1 with  $c_v(t)$, as claimed.    \end{proof}
\vspace{.5em}

   Hence $\widehat{\mc{E}}(v) $ consists of the initial vectors of $ \cap _T \ker \Omega _v ^T$ which is the intersection of the descending set of vector subspaces $\ker \Omega _v ^T$ as T increases.  Hence  there is a smallest number $T(v) < \infty$ such that $\widehat{\mc{E}} (v) $ consist of the initial vectors of  $  \ker \Omega _v ^T$ for all $T> T(v)$.

\begin{proposition}
 $\widehat{\mc{E}}$ is smooth on $\widehat{\mc{R}}$.   In particular,  $\widehat{\mc{E}}$ is smooth on $W^s(v)$ (resp. $W^u(v)$ ) where $v \in \widehat{\mc{R}}$ is forward (resp. backward) recurrent.
\end{proposition}

\begin{proof} Let $v \in \widehat{\mc{R}}$, and let $v_n \rightarrow v$.  We may assume that  $v_n \in \widehat{\mc{R}}$ since $\widehat{\mc{R}}$ is open.  Note that $T(v_n) < T(v) +1$ for all large enough $n$.  Otherwise, we could find $\hrk M +1$ many orthonormal parallel fields along $c_{v_n}$ which make curvature -1 with  $c_{v_n} (t)$ for $- T(v) - 1 < t < T(v)  + 1$.   Taking limits, we find $\hrk M +1$ many orthonormal parallel fields along $c_v$ which make curvature -1 with  $c_{v_n} (t)$ for $- T(v) - 1 < t < T(v)  + 1$. Therefore there exists a neighborhood $U\subset \widehat{\mc{R}}$ of $v$ such that $T(u)<T(v)+1$ for all $u\in U$. Since the quadratic forms $\Omega _w ^{T(v)+1}$ are smooth on the neighborhood $U$ of $v$, we see that the distribution is smooth on $\widehat{\mc{R}}$. 

The last claim is immediate from smoothness on $\widehat{\mc{R}}$ and Lemma \ref{recurrent}.
\end{proof}

\section{Maximal Lyapunov exponents and hyperbolic rank}      \label{sec:lyapunov}

   The geodesic flow $g_t: SM \to SM$ preserves the Liouville measure $\dvol$ on $SM$, and is ergodic.  Hence  Lyapunov exponents are defined and constant almost everywhere w.r.t. $\dvol$.  We recall that they measure the exponential growth rate of tangent vectors to $SM$ under the derivative of $g_t$.  As is well-known, double tangent vectors to M correspond in a 1-1 way with Jacobi fields $J(t)$,  essentially since  $J(t)$ is uniquely determined by the initial condition $J(0), J'(0)$.   Moreover we have
$$ ( D g_t) (J(0), J'(0))  =  (J(t), J'(t)). $$
Thus we can work with Jacobi fields rather than double tangent vectors whenever convenient.  We note that stable (resp. unstable) vectors for $g_t$ correspond to Jacobi fields which tend to 0 as $t \rightarrow \infty$  (resp. as 
$t \rightarrow -\infty$).

   If $-1\leq K \leq 0$, then all Lyapunov exponents of unstable Jacobi field along the geodesic flow for any invariant measure are between 0 and 1, cf. e.g. \cite[ch. IV,Prop. 2.9]{BallmannBook}.  
Similarly, if $K \leq -1$, all Lyapunov exponents have absolute value at least 1.   We want to understand the extremal case better.  We suppose $K \geq -1$ throughout.

\begin{lemma}\label{ortho-invar}
Let $\widehat{\mc{E}}(v)^\perp$ be the orthocomplement (with respect to the Riemannian metric on $M$) of $\widehat{\mc{E}}(v)$. Then $\Lambda(v,t)$ sends $\widehat{\mc{E}}(v)^\perp$ to $\widehat{\mc{E}}(g_tv)^\perp$.
\end{lemma}

\begin{proof} 
Indeed,  let $E_1(t),\dots, E_{n-1}(t)$ be  a choice of parallel orthonormal fields along $g_t v$ and perpendicular to $g_tv$  such that $\{E_1(t),\dots, E_k(t)\}$ forms a basis of  $ \widehat{\mc{E}}(g_tv)$. For any $w\in v^\perp$, the formula for an unstable Jacobi field becomes
\[\Lambda(v,t)w=\sum_i f_i(t)E_i(t).\]
Setting $a_{ij}=\inner{R(g_tv, E_i(t))g_tv,E_j(t)}$, the Jacobi equation is equivalent to 
\[f''_j(t)+\sum_i a_{ij}(t)f_i(t)=0.\]
Since  $e^t E_i(t)$ is an unstable Jacobi field for $i \leq k$ and the $\{E_i(t)\}$ are orthonormal,
\[
\inner{ R(g_t v, E_i (t) )g_t v, E_j (t)}=-\inner{E_i(t),E_j(t)} = -\delta^j_i,
\]
for all $i\leq k$ and any $j\leq n-1$. By the symmetries of the curvature tensor, $a_{ij}=a_{ji}$ and so 
we also have $a_{ji}(t)=a_{ij}(t)=-\delta_i^j$ for either $i\le k$ or $j\le k$. It follows that for all $t\in \R$ and all $i\le k$
\[
0=f_i''(t)+\sum_j a_{ij}(t)f_j(t)=f_i''(t)-f_i(t).
\]
Since $\Lambda(v,t)w$ is unstable, $\lim\limits_{t\to -\infty}f_i(t)=0$ for all $i$. If $w\in\widehat{\mc{E}}(v)^\perp$, then $f_i(0)=0$ for all $i\leq k$. These two conditions together imply $f_i(t)=0$ for all $t\in \R$ and $i\leq k$. Hence, $\Lambda(v,t)$ leaves $ \widehat{\mc{E}}^\perp$ invariant.
\end{proof}

\begin{lemma}\cite[ch. IV,Prop. 2.9]{BallmannBook}  \cite[Lemma 2.3]{Connell03}\label{BallCon-lem}
$\norm{\Lambda(v,t)w}\leq \norm{w}e^{t}$ for all $t\geq 0$. The equality holds at a time $T\in \R$ if and only if the sectional curvature of the plane spanned by $\Lambda(v,-t)w$ and $g_{t}v$ is -1 for all $0\leq t\leq T$ if and only if $\Lambda(v,t)w = \norm{w} e^t W(t)$ where $W(t)$ is parallel for all $0\leq t\leq T$.
\end{lemma}
\begin{proof}

   By the Rauch Comparison Theorem, $\norm{\Lambda(v,t)w}\leq \norm{w}e^{t}$ and $\norm{ \Lambda ' (v,t)w}\leq \norm{ \Lambda(v,t)w}$ for all $t\geq 0$ (cf. \cite[ch. IV, Prop. 2.9]{BallmannBook} which states a similar result for  stable Jacobi fields).   If equality holds at time $T>0$ then $\norm{\Lambda(v,t)w}=e^t\norm{w}$ for all $0\le t\le T$. Indeed, should $\norm{\Lambda(v,t_0)w }  <  e^{t_0}\norm{w}$ for some $0<t_0 <T$, then we get a contradiction since 
$$\norm{\Lambda(v,T)w }  \leq \norm{\Lambda(\Lambda(v,t_0)w , T-t_0)}  \leq e^{T-t_0} \norm{\Lambda(v,t_0)w }< e^{T-t_0}  e^{t_0}\norm{w} = e^T \norm{w}.$$

   Therefore the vector field $W(t)$ for which $\Lambda(v,t)w=  \norm{w}e^{t}W(t)$ is a field of norm $1$. Hence $\inner{W(t),W'(t)}=0$ and we have
\[
\norm{w}e^t (1+ \norm{W' (t) } ^2) ^{1/2} =  \norm{w} \norm{(e^t W(t) + e^t W' (t))}= \norm{\Lambda '(v,t)w }  \leq  \norm{\Lambda(v,t)w } = e^t \norm{w},
\]
by the estimate above on the derivative of the unstable Jacobi field.   We see that $W'=0$, i.e. $W$ is parallel as desired.
That the sectional curvature between $W(t)$ and the geodesic is -1 now follows from the Jacobi equation.
\end{proof}

   By covering the unit tangent bundle with countable base of open sets that generate the topology, and applying the ergodic theorem to the Liouville measure, there is a full measure set of unit tangent vectors that comes back to all its neighborhoods with positive frequency.

   The argument in the next lemma is similar to that of Lemma 3.4 of \cite{Ballmann85} 
and Proposition 1 of \cite{Hamenstadt91a}, but for the setting of a lower curvature bound of $-1$.

\begin{lemma}   \label{maxLyap}
Suppose $v\in\mc{R}$ returns with positive frequency to all its neighborhoods under $g_t$. Then for $w\in  {\mc{E}}(v)^\perp=\widehat{\mc{E}}(v)^\perp$, the unstable Jacobi field $\Lambda(v,t)w$ has Lyapunov exponent strictly smaller than 1. We have a similar statement for stable Jacobi fields of Lyapunov exponent -1.  
\end{lemma} 

\begin{proof}
Let $T>0$ be such that the dimension of parallel vector fields making curvature -1 with $g_tv$ for all $0\le t\le T$ is $k=\hrk(M)$, i.e. $k=\dim  {\mc{E}}(v)$ since $v\in\mc{R}$ (\ref{strong=weak}). 

    Pick $w_0\in  {\mc{E}}(v)^\perp$ that minimizes $\{\frac{\norm{w}}{\norm{\Lambda(v,T)w}}:w\in   {\mc{E}}(v)^\perp\}$. By Lemma \ref{BallCon-lem}, we have that $\norm{\Lambda(v,T)w_0}\le e^T\norm{w_0}$. Suppose that we have  the equality $\norm{\Lambda(v,T)w_0}= e^T\norm{w_0}$. Then by Lemma \ref{BallCon-lem}, the parallel field of $w_0$ along $g_t v$ makes curvature -1 with $g_tv$ for all $0\le t\le T$. Since $w_0\in  {\mc{E}}(v)^\perp$,  the space of parallel fields  making curvature -1 with $g_tv$ for all $0\le t\le T$ has dimension at least $\dim  {\mc{E}}(v) +1=k+1$, a contradiction. Therefore $\norm{\Lambda(v,T)w_0}< e^T\norm{w_0}$.

   Let $\epsilon >0$ be such that $\norm{\Lambda(v,T)w_0}= (1-2\epsilon)e^T\norm{w_0}$. By continuity, we can choose a neighborhood $U\subset \mc{R}$ of $v$ such that for all $u\in U$ and $w\in  {\mc{E}}(u)^\perp$, we have the estimate $\norm{\Lambda(v,T)w}\le (1-\epsilon)e^T\norm{w}$.

   Since $g_tv$ visit $U$ with a positive frequency, there are $\delta>0$ and $T_0>0$ such that for all $S>T_0$
\[|\{t\in[0,S]:g_{t}v\in U\}|>\delta S.\]

   Now suppose that $w\in v^\perp\cap  {\mc{E}}(v)^\perp$. We note that by Lemma \ref{ortho-invar},  since ${\mc{E}}(v)^\perp=\widehat{\mc{E}}(v)^\perp$ on $\mc{R}$, $\Lambda(v,S)w\in  {\mc{E}}(g_Sv)^\perp$ for all $S>0$. Then $\norm{\Lambda(v,S)w} \le e^{S}(1-\epsilon)^{[\frac{\delta S}{T}]}\norm{w}$. It follows that $\Lambda(v,t)w$ has Lyapunov exponent strictly smaller than 1.
\end{proof}

     We remark that the argument in the last proof only provides information that the unstable Jacobi fields come from parallel fields in forward time.  This forced us to introduce both sets ${\mc{R}}$ and $\widehat{\mc{R}}$ and use the equality of ${\mc{E}}$ and $\widehat{\mc{E}}$ on ${\mc{R}}$.

   Recall that  there is a  contact form $\theta$ on $SM$ invariant under the geodesic flow. Its exterior derivative $\omega = d \theta$ is a symplectic form on stable plus unstable distribution $E^s + E^u$.  Also $\theta$ and hence $\omega$ are invariant under the geodesic flow, and thus every Oseledets space $E_{\lambda}$ with Lyapunov exponent $\lambda$ is $\omega$-orthogonal to all $E_{\lambda'}$ unless $\lambda'=- \lambda$. Since $\omega$ is non-degenerate, $\omega$ restricted to $E_\lambda\times E_{-\lambda}$ is also non-degenerate for each $\lambda$. Note that 
$\widehat{\mc{E}}$ gives rise to unstable Jacobi fields with Lyapunov exponent 1.

    This immediately gives the following

\begin{corollary} 
The maximal Lyapunov spaces $E_1$ can be extended to be a $C^1$ distribution $E_1^u$ on the regular set $\widehat{\mc{R}}$. The orthogonal complement   $(E_1^u)^{\perp} \cap E^s$ w.r.t. $\omega$ is defined and $C^1$ on the same set $\widehat{\mc{R}}$ and equals   $\oplus _{-1 < \lambda <0} E_{\lambda}$ almost everywhere.   The analogous statements hold for $E_{-1}$ (yielding $E_1^s$) and $E_{-1}^\perp\cap E^u=\oplus _{0 < \lambda <1} E_{\lambda}$ a.e. as well.
\end{corollary}

   We will call the spaces $E_{<1}^s:=(E_1^u)^{\perp}\cap E^s$ and $E_{<1}^u:=(E_{1}^s)^{\perp}\cap E^u$  the {\em extended slow  stable and unstable} subspaces. Similarly we call $E_1^{s}$ and $E_1^{u}$ the {\em extended fast stable and unstable} subspaces.

\begin{proof}  On $\widehat{\mc{R}}$, $\widehat{\mc{E}}$ is defined and smooth.  On $\mc{R}$, $\mc{E}$ is also defined and smooth. Moreover the distribution $\mc{E}$ agrees with $\widehat{\mc{E}}$  on $\mc{R}$.  

   By Lemma \ref{maxLyap}, on the set  $\Omega=\{v\in \mc{R}, v \text{ is forward recurrent under } g_tv\}$,  $E_1$ agrees with  the lift to unstable Jacobi fields of $\mc{E}$  on $TSM$, i.e. $w\in \mc{E}(v)$ is identified with the Jacobi field $\Lambda(v,t)w$. The set $\Omega$ has full measure.  Hence $E_1$ extends smoothly on $\widehat{\mc{R}}$ to a distribution $E_1^u$.

   Now take the orthogonal  complement (w.r.t.~ the form $\omega$) to $E_1^u$,  $(E_1^{u})^{\perp} \cap E^s$, on $\widehat{\mc{R}}$ in the stable distribution.
 Since $E^s$ is $C^1$, and $E_1^u$ is even $C^{\infty}$, $(E_1^u)^{\perp} \cap E^s$ is $C^1$. 
 
   Since $\omega$ pairs Lyapunov spaces where defined a.e. on $\widehat{\mc{R}}$, $\oplus _{-1 < \lambda <0} E_{\lambda}\subset (E_1^u)^\perp\cap E^s$. Since $\omega$ is nondegenerate, the dimension of the latter subspace is exactly $n-1-\hrk(M)$ everywhere on $\widehat{\mc{R}}$, and hence they agree.

   A similar argument applies to  $E_{-1}$ and its perpendicular complement  w.r.t. $\omega$ in the unstable subspace where now we use forward recurrent vectors.   \end{proof}

\section{Slow Stable Spaces and Integrability}  \label{sec:kanai}

    In the tangent bundle $T\, SM$ of the unit tangent bundle, consider the subset $T\widehat{\mc{R}}\subset T\,SM$, which is the  union of tangent fibers of $SM$ at points in $\widehat{\mc{R}}$. On $T\widehat{\mc{R}}$, there is a $C^1$ decomposition $E_{1}^s+E_{<1}^s + E^0+E_{1}^u+E_{<1}^u$, where $E_1^{s/u}, E_{<1}^{s/u}$ denote  the extended stable/unstable fast and slow Lyapunov exponent distributions respectively   defined in the last section.

   We will define a special connection for which this decomposition is parallel, and use that to argue integrability of the slow unstable direction.   Such connections were introduced by Kanai to study geodesic flows with smooth stable and unstable foliations in \cite{kanai88}.  Our particular construction is motivated by that of Benoist, Foulon and Labourie in \cite{BFL90} where they classify contact Anosov flows with smooth Oseledets' decomposition.  We refer to \cite[Definition 2.49 and Proposition 2.58]{Gallot-Hulin-Lafontaine} for the basic facts on affine connections we will need.

    We recall the formula for the contact 1-form $\theta$:   $\theta_{(x,v)}(\xi)=<v,\xi^0>$, where $(x,v)\in SM$ and $\xi\in T_{(x,v)}SM$, where $\xi_0=d\,pt (\xi)$. 
Then  the 2-form $d \theta$ becomes
$$d\theta(\xi,\eta)=<\xi_1^u,\eta_1^s>+<\xi_{<1}^u,\eta_{<1}^s>-<\xi_1^s,\eta_1^u>-<\xi_{<1}^s,\eta_{<1}^u>,$$
where the indices indicate appropriate components when we decompose $\xi$ or $\eta$ w.r.t.~ the decomposition $E_{1}^s+E_{<1}^s + E^0+E_{1}^u+E_{<1}^u$.

   We let $\mc{X}$ denote the geodesic spray, i.e., the generator of the geodesic flow which is the vector field belonging to $E_0$ obtained by lifting unit tangent vectors of $M$ to $TSM$ horizontally.
  
\begin{proposition}\label{prop:conn}
There exists a unique connection $\nabla$ on $T\widehat{\mc{R}}$ such that
\begin{enumerate}
\item $\nabla \theta =0$, $\nabla d\theta=0$, and $\nabla E_0\subset E_0, \nabla E_{i}^{s/u}\subset E_i^{s/u}$ for $i\in \{1,<1\}$.
\item For any sections $Z_{1}^s,Z_{<1}^s, Z_{1}^u,Z_{<1}^u$ of $E_{1}^s,E_{<1}^s,E_{1}^u,E_{<1}^u$ respectively, we have for $i,j\in \{1,<1\}$
\[\nabla_{Z_{i}^s}Z_{j}^u=p_{E_{j}^u}([Z_{i}^s,Z_{j}^u]),\]
\[\nabla_{Z_{i}^u}Z_{j}^s=p_{E_{j}^s}([Z_{i}^u,Z_{j}^s]),\]
\[\nabla_{\mc{X}}Z_{i}^{s/u}=[\mc{X},Z_{i}^{s/u}],\]
where the $p_{E_{j}^{s/u}}$ are the projections to the $E_j^{s/u}$ subspaces.
\end{enumerate}
In addition, $\nabla$ is invariant under the geodesic flow $g_t$.
\end{proposition}
\begin{proof}
We note that $\nabla d\theta = 0$ is equivalent with $Wd\theta(Y,Z)=d\theta(\nabla_WY,Z)+d\theta(Y,\nabla_WZ)$ for any vector fields $W,Y,Z$. And $\nabla \theta = 0$ is equivalent with $\theta(\nabla_YZ)=Y\theta(Z)$ for any vector fields $Y,Z$. Thus $\theta(\nabla_Y \mc{X})=0$ and $d\theta(\nabla_Y  \mc{X},Z)=0$ for any vector fields $Y,Z$. It follows that $\nabla \mc{X}=0$.    Furthermore, given a $C^1$ function $f:\widehat{\mc{R}} \to \R$, we set $\nabla _Y (f  \mc{X}) = Y(f)  \mc{X}$.    

   Moreover, $\nabla_{Z_{i}^u}Z_{j}^u$ is uniquely determined by the condition $\nabla E_i^u\subset E_i^u$ and the equality 
\[Z_{i}^ud\theta(Z_{j}^u,Z^s)=d\theta(\nabla_{Z_{i}^u}Z_{j}^u,Z^s)+d\theta(Z_{j}^u,\nabla_{Z_{i}^u}Z^s),\]
for arbitrary section $Z^s$ of $E^s$ and $i,j\in \{1,<1\}$. Similarly  $\nabla_{Z_{i}^s}Z_{j}^s$ is uniquely determined.
By linearity,   we have defined $\nabla _Y Z$ for all vector fields $Y, Z$.   

   It is now easy to check that $\nabla$ satisfies the properties of a connection on $\widehat{\mc{R}}$ (cf. e.g. \cite[Definition 2.49]{Gallot-Hulin-Lafontaine}).   
That $\nabla$ is invariant under the geodesic flow $g_t$ follows from the construction.  Indeed, the slow and fast stable and unstable spaces are invariant under $g_t$, $(g_t)_* ([Y, Z] ) = [ (g_t)_* Y, (g_t)_*  Z]$ and $ \mc{X}$ is invariant under $g_t$ by definition.
\end{proof}

     The next lemma is basically well-known (cf.   e.g. \cite[Lemma 2.5]{BFL90}).   Since our connection is only defined on a dense open set and not necessarily bounded we outline the  proof.
Since Liouville measure is ergodic for the geodesic flow $g_t$ on $SM$, the Lyapunov exponents $\gamma _i$ are defined and constant on a  $g_t$-invariant full measure set $\Sigma$ in $\widehat{\mc{R}}$.  We can assume in addition that all $v \in \Sigma$ are forward and backward recurrent for $g_t$, and that the Oseledets decomposition  $T_v \mc{R} = \oplus E_{\gamma _i}$  into Lyapunov subspaces $E_{\gamma _i}$ is defined on $\Sigma$.  Thus if $Z_i \in E_{\gamma _i}$, the forward and backward Lyapunov exponents are defined and equal to $\gamma _i$. 

\begin{lemma}\label{lem:tensor} Let $v \in \Sigma$.  If $K$ is a geodesic flow invariant tensor and $Z_1,\dots,Z_k$ are vectors in $T_v \widehat{\mc{R}}$ with $Z_i \in E_{\gamma _i}$, then  $K(Z_1,\dots,Z_k)$ is either zero or has Lyapunov exponent $\gamma_1+\dots+\gamma_k$.
\end{lemma}

\begin{proof}
There is a neighborhood $U$ of $v$ and $C>0$ such that $\norm{K(Y_1,\dots,Y_k)}\le C \norm{Y_1}\cdots \norm{Y_k}$, for any vectors  $Y_1,\dots,Y_k$ with footpoints in the  neighborhood $U$.  Suppose that $K(Z_1,\dots,Z_k)\neq 0$. If  $g_t(v)\in U$  for some $t>0$ then
\[\norm{D_vg_t K(Z_1,\dots,Z_k)}=\norm{K(D_vg_tZ_1,\dots,D_vg_tZ_k)}\le C \norm{D_vg_tZ_1} \cdots \norm{D_vg_tZ_k}.\]
Thus,
\[\frac{1}{t}\log(\norm{D_vg_t K(Z_1,\dots,Z_k)})\le \frac{1}{t}(\log(C)+log(\norm{D_vg_tZ_1})+\dots +\log(\norm{D_vg_tZ_k})).\]
   Since $v$ is forward recurrent, therre will be a sequence of times $ t \rightarrow \infty$  with $g_t(v)\in U$.  Thus  the forward Lyapunov exponent of  $K(Z_1,\dots,Z_k)$  is    at most  $\gamma_1+\dots+\gamma_k$.  Hence  $K(Z_1,\dots,Z_k)$ cannot have nonzero components in  $E_{\gamma}$ if $\gamma > \gamma_1+\dots+\gamma_k$.

Similarly, if $g_s(v)\in U$ for some $s<0$ then
\[\frac{1}{s}\log(\norm{D_vg_s K(Z_1,\dots,Z_k)})\ge \frac{1}{s}(\log(C)+log(\norm{D_vg_sZ_1})+\dots +\log(\norm{D_vg_sZ_k})).\]
  Since $v $ is backward recurrent, arguing as above,   the backward Lyapunov exponent of  $K(Z_1,\dots,Z_k)$  is    at least  $\gamma_1+\dots+\gamma_k$.
Hence $K(Z_1,\dots,Z_k)$ cannot have nonzero components in  $E_{\gamma}$ if $\gamma  < \gamma_1+\dots+\gamma_k$.
 
\end{proof}

    Recall that the connection $\nabla$ is only $C^1$, and only defined on $\widehat{\mc{R}}$.  This means that the torsion tensor is only a $C^0$-tensor, and the curvature tensor is not defined.  However, slow and fast stable and unstable distributions are smooth on stable and unstable manifolds in $\widehat{\mc{R}}$.  Hence the restriction of $\nabla$ to stable or unstable manifolds is also smooth by the construction of $\nabla$.   In particular,  the curvature tensor of $\nabla$ restricted to stable or unstable manifolds is well defined.  

\begin{corollary}\label{conn-flat}
The torsion and curvature tensors of $\nabla$ restricted to the slow Lyapunov distributions $E_{<1}^{s/u}$ 
 and also each stable/unstable space $E^{s/u}$ are zero.
\end{corollary}
\begin{proof}   Since $\nabla$ is geodesic flow invariant, so are the torsion and curvature tensors.   In strict $\frac 1 4$-pinched manifolds,  the ratio of any two Lyapunov exponents lies in  $(\frac 1 2, 2)$. Thus this corollary follows at points of $\Sigma$ immediately from the previous lemma and the strict $\frac 1 4$-pinching condition. Since $\Sigma$ is of full measure, and therefore dense, the statements hold everywhere on $\widehat{\mc{R}}$ by continuity.
\end{proof}

\begin{corollary}\label{cor:integ}
The slow unstable Lyapunov  distribution  $E_{<1}^{u}$ is integrable.
\end{corollary}

\begin{proof} 
The slow unstable Lyapunov  distribution $E_{<1}^{u}$ is invariant under the parallel transport by $\nabla$, by construction of $\nabla$. Since $\nabla$ is flat, parallel transport is independent of path. Thus we can choose canonical local parallel $C^1$ vector fields tangent to and spanning the distribution. On the other hand, since the restriction of torsion on unstable leaves is zero we have that the commutators of these vector fields are zero. By the Frobenius Theorem for $C^1$ vector fields, \cite[Theorem 1.1 Chapter 6]{Lang95}, the distribution is integrable.
\end{proof}

As usual we will consider the $\pi_1(M)$-lifts of the stable and unstable manifolds by the same notation in $S\til{M},$ and we will work in $SM$ or $S\til{M}$ as appropriate without further comment. 
Given $v\in S\til{M}$ , the map $\pi_v: W^u(v)\to \partial \til{M}-\{c_v(-\infty)\}$, defined by $\pi_v(w)=c_w(\infty)$, is a $C^1$ diffeomorphism. For $w\in W^s(v)$, the {\em stable holonomy} is defined as 
$$h_{v,w}:W^u(v)-\{\pi_v^{-1}(c_w(-\infty)\}\to W^u(w)-\{\pi_w^{-1}(c_v(-\infty))\}.$$
Note that $ h_{v,w} (x)$ is simply the intersection of the weak stable manifold
of $x$ with  $W^u(w)$.  In particular the stable holonomy maps are $C^1$. Indeed, the sectional curvatures of $M$ are strictly $\frac 14$-pinched and hence the weak
stable foliation is $C^1$ \cite{Hirsch75}. Moreover, the stable holonomy maps
$h_{a,b}$ are $C^1$ with derivative bounded uniformly in $d_{S\til{M}}(a,b)$ for
$b\in \cup_tg^t W^s(a)$. This follows from the fact that the unstable foliation
is uniformly transversal to the stable foliation, by compactness of $SM$. In fact, Hasselblatt \cite[Corollary 1.7]{Hasselblatt94} showed that the derivative is even \holder continuous.
  
We call a distribution {\em stable holonomy invariant}  if it is invariant under (the derivative map of) all holonomies $h_{v,w}$ for all $v\in S\til{M}$ and $w\in W^s(v)$. We will now adapt  an argument by Feres and Katok \cite[Lemma 4]{FeresKatok1990}.

\begin{lemma}  \label{holoNomy}
The slow unstable spaces $E_{<1}^{u}\subset T\widehat{\mc{R}}$ are stable holonomy invariant.
\end{lemma}

\begin{proof}
First consider $v\in \Sigma\cap \mc{R}$ and $w\in \cup_t g^tW^s(v)\cap\Sigma\cap
\mc{R}$. The distance between $g^tv$ and $g^tw$ remains bounded in forward time.
Hence the derivatives of the holonomy maps $h_{g^tv,g^tw}$ are uniformly bounded
for all $t\geq 0$. If $u\in E^u_{<1}(v)$, then $u$ has forward Lyapunov exponent
$\la<1$ since $v\in\Sigma$. It follows that the image vector $Dh_{v,w}(u)$ also
has forward Lyapunov exponent $\la<1$ and hence belongs to $E^u_{<1}(w)$. In
particular, $Dh_{v,w}E^u_{<1}(v)\subset E^u_{<1}(w)$. By continuity of
$Dh_{v,w}$ and of the extended slow space on $\widehat{\mc{R}}$ the same holds for all
$v,w\in \widehat{\mc{R}}$.
\end{proof}

   We follow ideas of Butler \cite{Butler15} to derive:

\begin{corollary}
The slow unstable distributions are  trivial.
\end{corollary}

\begin{proof}

By the strict $\frac 14$-pinching, the boundary $\partial \til{M}$ of the universal cover admits  a $C^1$ structure for which the projection maps from points or horospheres are $C^1$ (\cite{Hirsch75}). By Lemma \ref{holoNomy}, the projection of the lifts of the slow unstable distribution is independent of the projection point on the horosphere. Using different horospheres we obtain a well-defined distribution on all of $\partial \til{M}$. Note that this distribution is also invariant under $\pi_1(M)$. 

By Corollary \ref{cor:integ}, this distribution is integrable and yields a $C^1$ foliation $\mc{F}$ on the boundary $\partial \til{M}$ which is also $\pi_1(M)$-invariant. Since there is a hyperbolic element of $\pi_1(M)$ which acts with North-South dynamics on $\partial \til{M}$, by Foulon \cite[{{\em  Corollaire}\rm }]{Foulon1994}, the foliation generated by this distribution has  to be trivial.     
\end{proof}
\vspace{.6em}

We are now ready to finish the proof of our main result.  
\vspace{.6em}

\begin{proof}[Proof of Theorem \ref{thm:main}] (strict $\frac14$-pinching case):  Since the slow unstable distribution is trivial,  all unstable Jacobi fields belong to $E^u _1$.  Hence  all sectional curvatures are $-1$ on $\widehat{\mc{R}}$.  Since $\widehat{\mc{R}}$ is open dense in $SM$,  it follows that all sectional curvatures are $-1$.   \end{proof}

\section{Non-Strictly $\frac{1}{4}$-Pinched Case}\label{sec:non-strict}

In this section we extend the proof of the main theorem to the non-strictly $\frac14$-pinched curvature case.

First, consider the set $\mc{O}\subset SM$ of vectors whose smallest positive Lyapunov exponent is $\frac12$. If $\mc{O}$ has positive Liouville measure then Theorem 1.3 of \cite{Connell03} implies $M$ is locally symmetric and our theorem holds. Hence we may assume $\mc{O}$ has measure $0$ and there is a flow invariant full measure set $\mc{P}$ and a $\nu>0$ such that for all $v\in\mc{P}$ the unstable Lyapunov exponents satisfy $\frac 1 2 +\nu < \chi_i^+(v) \le 1$. 

Note that, unlike in the strict quarter-pinched case, we cannot immediately use the vanishing of the torsion of the generalized Kanai connection established in Proposition \ref{prop:conn}. Indeed,  the construction of the generalized Kanai connection used that both stable and unstable distributions are $C^1$ on $SM$ which we do not a priori know in our case.

Instead, we replace the generalized Kanai connection with a similar one assembled from the flow invariant system of measurable affine connections on unstable manifolds constructed  by Melnick in \cite{Melnick16}. The connections are defined on whole unstable manifolds but they are only defined for unstable manifolds $W^u(v)$ for $v$ in a set of full measure.  Moreover, the transversal dependence is only measurable. Mark that we have switched from Melnick's usage of stable manifolds to unstable manifolds.

Following the notation in \cite{Melnick16}, let her $\mc{E}$ be the smooth tautological bundle over $SM$ whose fiber at $v$ is $W^u(v)$. We consider the cocycle $F_v^t$ which is $g^t$ restricted to $W^u(v)$. The ratio of maximal to minimal positive Lyapunov exponents lies in $[1,2)$, and hence the integer $r$ appearing in  Theorem 3.12 of \cite{Melnick16} is $1$. This theorem then reads in our notation as:

\begin{lemma}
There is a full measure flow-invariant set $\mc{U}\subset SM$ where there is a smooth flow-invariant flat connection $\nabla$ on $TW^u(v)$ for $v\in \mc{U}$.
\end{lemma}

Now we build a connection on vector fields tangent to the slow unstable distribution $E^{u}_{<1}$ on $W^u(v)$ for $v\in \mc{U}$. We emphasize that we do not assume integrability of the slow unstable distribution.  We just construct a connection on sections of the vector bundle given by the slow unstable distribution. More specifically on slow unstable distribution we have the following.

\begin{lemma}
On each unstable leaf $W^u(v)$ for $v$ in a full measure flow invariant subset $\mc{Q}\subset \mc{U}\cap\mc{P}\cap\widehat{\mc{R}}$, there exists a torsion free and flow invariant \cout{flat }connection 
\[
\nabla^{<1}:TW^u(v)\times C^1(W^u(v),E^{u}_{<1})\to C^1(W^u(v),E^{u}_{<1})
\] 
on $E^{u}_{<1}$. Moreover the restriction of the connection to $E^{u}_{<1}\mc{X}$ is torsion free\cout{ and flat}.
\end{lemma}

\begin{proof}
Recall that the distribution $E^{u}_{<1}$ is smooth on $W^u(v)$ for $v\in \widehat{\mc{R}}$. 

Given $X\in TW^u(v)$ and $Y\in C^1(W^u(v),E^{u}_{<1})$ we define the covariant derivative $\nabla^{<1}_X Y$ to be the section in $C^1(W^u(v),E^{u}_{<1})$ given by projection of the Melnick connection, 
\[ \nabla^{<1}_X Y:=\proj_{E^{u}_{<1}}\nabla_{X}Y.\]

Note that this operator is $\R$-bilinear in $X$ and $Y$ since projections are linear, and for $f\in C^1(W^u(v))$ since scalar functions commute with projection we have 
\beq 
\nabla^{<1}_{fX}Y&=\proj_{E^{u}_{<1}}f\nabla_{X}Y=f\proj_{E^{u}_{<1}}\nabla_{X}Y=f \nabla^{<1}_{X}Y\\
\nabla^{<1}_{X}fY&=\proj_{E^{u}_{<1}}X(f)Y+ f\nabla_{X}Y=X(f)Y +f \nabla^{<1}_{X}Y.\
\eeq
Here we have used that $\proj_{E^{u}_{<1}}Y=Y$.
Hence $\nabla^{<1}$ is $C^1(W^u(v))$-linear in $X$, and satisfies the derivation property of connections.

For $v\in \mc{P}\cap\mc{U}$, $X,Y\in C^1(W^u(v),E^{u}_{<1})$ the torsion tensor $T(X,Y)=\nabla^{<1}_X Y-\nabla^{<1}_Y X-[X,Y]$ is indeed a tensor due to the derivation property of the connection and bracket where we take the bracket of vector fields in $W^u(v)$. 

Next we show that $\nabla^{<1}$ is torsion free. Since $[X,Y]$ and $\nabla^{<1}$ are
invariant under $Dg_t$, so is $T(X,Y)$. Also, since $v\in \mc{P}$ the sum of any
two Lyapunov exponents lies in $(1,2]$. By Fubini, and absolute continuity of
the $W^u$ foliation, we may choose
$\mc{Q}\subset\mc{U}\cap\mc{P}\cap\widehat{\mc{R}}$ to be an invariant full measure
set where for each $v\in \mc{Q}$ a.e. $w\in W^u(v)$ is forward and backward
recurrent. Now we can apply Lemma \ref{lem:tensor} to almost every $w\in W^u(v)$
to obtain that $\nabla^{<1}$ is torsion free on a dense subset and hence on all
of $W^u(v)$.

\end{proof}

\begin{corollary}
The slow unstable Lyapunov  distribution $E^{u}_{<1}$ is integrable on every leaf $W^u(v)$ for $v\in \mc{Q}$.
\end{corollary}

\begin{proof}
For $v\in \mc{Q}$, and $X,Y\in C^1(W^u(v),E^{u}_{<1})$ the vanishing of the torsion tensor implies $T(X,Y)=0=\nabla^{<1}_X Y-\nabla^{<1}_Y X-[X,Y]$.
However, by definition $\nabla^{<1}_X Y$ and $\nabla^{<1}_Y X$ belong to $E^{u}_{<1}$, and therefore so does $[X,Y]$. In particular, $E^{u}_{<1}$ is integrable.
\end{proof}

The above corollary gives us well defined slow unstable foliations on almost every $W^u(v)$. Next we will show that these foliations are invariant under stable holonomy. This is substantially more difficult in the non-strict $\frac14$-pinched case since the unstable holonomy maps a priori are not known to be $C^1$.

To simplify notation, we use $Dg_{t,v}$ for the derivative of $g_t$ at $v$ restricted to $E^u(v)$. Since $M$ is strictly $\frac14+\delta$-pinched for any $\delta>0$, Corollary 1.7 of \cite{Hasselblatt94} implies the following.  
\begin{lemma}\label{lem:holder}
The foliations $W^s$ and $W^u$ are $\alpha$-\holder for all $\alpha<1$.
\end{lemma}

Now choose an $\alpha > 1 -\frac \nu 4$. As in Kalinin-Sadovskaya \cite[Section 2.2]{Kalinin13} we have local linear identifications $I_{vw}:E^u(v)\to E^u(w)$ which vary in an $\alpha$-\holder way on a neighborhood of the diagonal in $SM\times SM$. We also have that $Dg_t$ is an $\alpha$-\holder cocycle, since it the restriction of the smooth $Dg_t$ to an $\alpha$-\holder bundle. In other words, with respect to these identifications, we have 
\be\label{eq:holder}
\norm{Dg_{t,v} -I_{g_tv,g_tw}^{-1}\circ Dg_{t,w}\circ I_{v,w}}\le C(T_0) d(v,w)^\alpha
\ee
for any $T_0>0$ and all $t\le T_0$.

We aim to show convergence of $Dg_{t,w}^{-1}\circ I_{g_tv,g_tw}\circ Dg_{t,v}$ for $w\in W^s(v)$ locally and both are in some good set.

For a.e. $v\in SM$, we let $T(v)=\inf\{s>0: \frac 1 t \log \norm{Dg_{t,v}\zeta}> \frac 1 2 +\frac \nu 2 \text{ for all } t>s\text{ and for all } \zeta\in E^u(v)\}$.

The following lemma can be found in Kalinin-Sadovskaya \cite{Kalinin13} under a hypothesis of uniform bunching. We have a similar statement under a nonuniform hypothesis. This provides the morale for Lemma \ref{fast-partial-derivative}, the result we will actually use.

\begin{lemma}\label{lem:converge}
Let $v\in SM$ with all Lyapunov exponents satisfying $\frac 1 2 +\nu < \chi_i^+(v) \le 1$, and let $w\in W^s(v)$. Then the limit $H_{v,w}=\lim\limits_{t\to+\infty}Dg_{t,w}^{-1}\circ I_{g_tv,g_tw}\circ Dg_{t,v}$ converges.
\end{lemma}

\begin{proof}
Denote $H_{vw}^t=Dg_{t,w}^{-1}\circ I_{g_tv,g_tw}\circ Dg_{t,v}$. We show the convergence of $(H_{vw}^t)^{-1}=H_{wv}^t$ instead. 

Consider $t_1>T(v)$ such that the identifications $I_{g_{t_1}v,g_{t_1}w}$ are defined and have \holder dependence. Then for all such $t_1$ sufficiently large and $t<T_0$, we claim $\norm{H_{wv}^{t_1+t}-H_{wv}^{t_1}}$ is exponentially small in terms of $t_1$. Indeed,
$$H_{wv}^{t_1+t}-H_{wv}^{t_1}=Dg_{t_1,v}^{-1}\circ (Dg_{t,g_{t_1}v}^{-1}\circ I_{g_{t_1+t}w,g_{t_1+t}v} \circ Dg_{t,g_{t_1}w} -I_{g_{t_1}w,g_{t_1}v}) \circ Dg_{t_1,w}.$$
Thus,
$$\|H_{wv}^{t_1+t}-H_{wv}^{t_1}\|\le \| Dg_{t_1,v}^{-1}\|.\|Dg_{t_1,w}\|.\|Dg_{t,g_{t_1}v}^{-1}\circ I_{g_{t_1+t}w,g_{t_1+t}v} \circ Dg_{t,g_{t_1}w} -I_{g_{t_1}w,g_{t_1}v}\|.$$
We have $\| Dg_{t_1,v}^{-1}\norm{.}Dg_{t_1,w}\|< e^{(-\frac 1 2 -\frac \nu 2)t_1}e^{t_1}=e^{(\frac 1 2 -\frac \nu 2)t_1}$. On the other hand, by \eqref{eq:holder} applied to $v=g_{t_1}w$ and $w=g_{t_1}v$, 
\beq
\|Dg_{t,g_{t_1}v}^{-1}\circ &I_{g_{t_1+t}w,g_{t_1+t}v} \circ Dg_{t,g_{t_1}w} -I_{g_{t_1}w,g_{t_1}v}\| \\ 
&\le \| Dg_{t,g_{t_1}v}^{-1}\circ I_{g_{t_1+t}w,g_{t_1+t}v}\| \| Dg_{t,g_{t_1}w} -I^{-1}_{g_{t_1+t}w,g_{t_1+t}v}\circ Dg_{t,g_{t_1}v}\circ I_{g_{t_1}w,g_{t_1}v}\| \\
&\le C_1(T_0) C(T_0)d(g_{t_1}w, g_{t_1}v)^\alpha \le C(T_0)d(w,v)^\alpha e^{-\frac 1 2 \alpha t_1}.
\eeq
The last inequality holds because of the curvature condition and $w\in W^s(v)$. Here we have absorbed $C_1(T_0)$ into the generic constant $C(T_0)$. Combining inequalities, and since we choose $\alpha>1-\frac{\nu}4$, we get the estimate:
\begin{align*}\|H_{wv}^{t_1+t}-H_{wv}^{t_1}\|\le C(T_0)d(w,v)^\alpha e^{(\frac 1 2 -\frac \nu 2)t_1}e^{-\frac 1 2 \alpha t_1}& \le  C(T_0)d(w,v)^\alpha e^{(\frac 1 2(1-\alpha) -\frac \nu 2)t_1}\\ & \le  C(T_0)d(w,v)^\alpha e^{-\frac \nu 4 t_1}.\end{align*}
This shows the convergence of $H_{wv}^t$.
\end{proof}

Let $v,w\in \widehat{\mc{R}}$ be backward recurrent under $g_t$. It follows that $W^u(v)\subset \widehat{\mc{R}}$ and $W^u(w)\subset \widehat{\mc{R}}$. Recall that for any $\eta$ in the weak unstable manifold of $ w$, every vector in $E^u_1(\eta)$ has parallel translate making curvature -1 with $g_t\eta$ for all time $t\in \mathbb{R}$. Denote by $\proj_{E^u_1(w)}$ the orthogonal projection from $E^u(w)$ onto $E^u_1(w)$.
We define $H^{t,1}_{vw}=Dg_{t,w}^{-1}\circ \proj_{E^u_1(w)}\circ \ I_{g_tv,g_tw}\circ Dg_{t,v}$ (the superscript ``1'' indicates the fast subspace $E_1^u$ as before). 

\begin{lemma}\label{fast-partial-derivative}
We have the following
\begin{enumerate}
\item $H^{t,1}_{v,w}$ converges to a limit, denoted $H^{1}_{v,w}$ for every $v,w$ with $w\in W^s(v)$.
\item If $\xi \in E^u(v)$ with forward Lyapunov exponent $\chi(v,\xi)<1$ then $H^{1}_{vw}(\xi)=0$.
\item The operator norm $\norm{H^{1}_{vw}}$ is locally bounded as $v$ and $w$, in the same weak stable leaf, vary in a sufficiently small neighborhood in $\widehat{\mc{R}}\times\widehat{\mc{R}}$.
\end{enumerate}
\end{lemma}

\begin{proof}
For (1), we follow the mode of proof of Lemma \ref{lem:converge}. Let $\xi \in E^u(v)$. For $t_1$ large and $t<T_0$ consider
\begin{align*}
&(H_{vw}^{t_1+t,1}-H_{vw}^{t_1,1})(\xi)=\\ & Dg_{t_1,w}^{-1}\circ (Dg_{t,g_{t_1}w}^{-1}\circ \proj_{E^u_1(g_{t_1+t}w)}\circ I_{g_{t_1+t}v,g_{t_1+t}w} \circ Dg_{t,g_{t_1}v} -\proj_{E^u_1(g_{t_1}w)}\circ I_{g_{t_1}v,g_{t_1}w}) \circ Dg_{t_1,v}(\xi).
\end{align*}
Since 
\[(Dg_{t,g_{t_1}w}^{-1}\circ \proj_{E^u_1(g_{t_1+t}w)}\circ I_{g_{t_1+t}v,g_{t_1+t}w} \circ Dg_{t,g_{t_1}v} -\proj_{E^u_1(g_{t_1}w)}\circ I_{g_{t_1}v,g_{t_1}w}) \circ Dg_{t_1,v}(\xi) \in E^u_1(g_{t_1}v),\] 
we have that 
\begin{align*} &\|(H_{vw}^{t_1+t,1}-H_{vw}^{t_1,1})(\xi)\| \\ &= e^{-t_1}\| (Dg_{t,g_{t_1}w}^{-1}\circ \proj_{E^u_1(g_{t_1+t}w)}\circ I_{g_{t_1+t}v,g_{t_1+t}w} \circ Dg_{t,g_{t_1}v} -\proj_{E^u_1(g_{t_1}w)}\circ I_{g_{t_1}v,g_{t_1}w}) \circ Dg_{t_1,v}(\xi)\|\\
& \le e^{-t_1}\|(Dg_{t,g_{t_1}w}^{-1}\circ \proj_{E^u_1(g_{t_1+t}w)}\circ I_{g_{t_1+t}v,g_{t_1+t}w} \circ Dg_{t,g_{t_1}v} -\proj_{E^u_1(g_{t_1}w)}\circ I_{g_{t_1}v,g_{t_1}w})\|.\| Dg_{t_1,v}(\xi)\|\\
& \le \|\xi\|.\|(Dg_{t,g_{t_1}w}^{-1}\circ \proj_{E^u_1(g_{t_1+t}w)}\circ I_{g_{t_1+t}v,g_{t_1+t}w} \circ Dg_{t,g_{t_1}v} -\proj_{E^u_1(g_{t_1}w)}\circ I_{g_{t_1}v,g_{t_1}w})\|.\end{align*}
Since projection and $I_{v,w}$ are as regular as the underlying vector bundle, by Lemma \ref{lem:holder} we have a \holder estimate 
\begin{align*}&\|(Dg_{t,g_{t_1}w}^{-1}\circ \proj_{E^u_1(g_{t_1+t}w)}\circ I_{g_{t_1+t}v,g_{t_1+t}w} \circ Dg_{t,g_{t_1}v} -\proj_{E^u_1(g_{t_1}w)}\circ I_{g_{t_1}v,g_{t_1}w})\|\\& \le C(T_0)d(g_{t_1}w, g_{t_1}v)^\alpha\le C(T_0)d(w,v)^\alpha e^{-\frac 1 2 \alpha t_1}.\end{align*}
It follows that we get the convergence.

Next for (2), assuming $\chi(v,\xi)=1-\delta$ for some $\delta>0$ we have
\begin{align*}\|H^{t,1}_{v,w}(\xi)\| & = e^{-t}.\|\proj_{E^u_1(g_tw)}\circ I_{g_tv,g_tw}\circ Dg_{t,v}(\xi)\|\\
& \le e^{-t}e^{(1-\frac{\delta}{2})t}.\|\proj_{E^u_1(g_tw)}\circ I_{g_tv,g_tw}\| .\|\xi\|\le Ce^{-\frac{\delta t}{2}}\|\xi\|,\end{align*}
where $C$ is a constant that bounds norms of identifications between close enough points. It follows that $H^{1}_{v,w}(\xi)=0$.

Finally for (3), let $\eta_1\in W^u(v)$ and $\eta_2\in W^u(w)$ be vectors in a small neighborhood of $v$ and $w$ such that $\eta_2$ is in the weak stable manifold of $\eta_1$. For some sufficiently small neighborhood of $v$ and $w$, there is small time $t_{\eta_2} \in (-\delta,\delta)$ for some $\delta >0$ such that $g_{t_{\eta_2}}\eta_2 \in W^s(\eta_1)$. If $\xi\in E^u(\eta_1)$ has forward Lyapunov exponent 1 then
\beq
\|H^{t,1}_{\eta_1,\eta_2}(\xi)\|&=\|Dg_{{t+t_{\eta_2}},\eta_2}^{-1}\circ \proj_{E^u_1(\eta_2)}\circ \ I_{g_t \eta_1,g_{t+t_{\eta_2}}\eta_2}\circ Dg_{t,\eta_1}(\xi)\| \\
&\le e^\delta\|\proj_{E^u_1(\eta_2)}\circ \ I_{g_t \eta_1,g_{t+t_{\eta_2}}\eta_2}\|\norm{\xi}\le Ce^\delta\norm{\xi},
\eeq
for $t$ large enough. The claim then follows.
\end{proof}

We define a map $\mc{I}_{vw}:=\exp_w|_{E^u(w)}\circ I_{vw}\circ (\exp_v|_{E^u(v)})^{-1}$ from a neighborhood $V$ of $v$ in $W^u(v)$ into a neighborhood $W$ of $w$ in $W^u(w)$, where $\exp_w|_{E^u(v)}$ and $\exp_w|_{E^u(w)}$ denote exponential maps into neighborhoods of $v$ in $W^u(v)$ and $w$ in $W^u(w)$. 

Let $h^t_{vw}=g_{-t}\circ \mc{I}_{vw} \circ g_t: V\to W$. The maps $h^t_{vw}$ converges to $h_{vw}$ locally at $v$ as $t\to +\infty$. The map $h^t_{vw}$ is smooth and $Dh^t_{vw}=Dg_{t,w}^{-1}\circ I_{g_tv,g_tw}\circ Dg_{t,v}$.

\begin{lemma}\label{lem:holonomy2}
Let $v,w\in \mc{Q}$ be backward recurrent under $g_t$. Assume further that $v$ and $w$ are chosen that almost every vector of $W^u(v)$ and $W^u(w)$ are in $\mc{R}$ and forward recurrent under $g_t$. Then the holonomy $h_{v,w}$ maps slow unstable leaves in $W^u(v)$ to slow unstable leaves in $W^u(w)$. 
\end{lemma}

\begin{proof}
The image of the $C^1$ slow unstable foliation in $W^u(v)$ under $h_{v,w}$ is a $C^0$ foliation in $W^u(w)$. We will first show that $h_{v,w}$ maps slow unstable leaves to locally. 

Choose foliation charts for the slow unstable foliations around $v$ and $w$. Let $\mc{H}$ be the connected component containing $v$ of the intersection with the slow unstable leaf through $v$ with the chart, i.e. the plaque of $v$. Similarly, let $\mc{V}$ be the plaque of the fast unstable leaf containing $w$. Assume first that almost every vector in $\mc{H}$ is forward recurrent. After shrinking $\mc{H}$ if necessary, let $f:\mc{H}\to \mc{V}$ be defined by choosing $f(\eta)\in\mc{V}$ to be the intersection of the slow unstable leaf containing $h_{v,w}(\eta)$ and $\mc{V}$. We show that $f$ is differentiable on $\mc{H}$, and its derivative is exactly $H^1_{v,w}$. Indeed, since projection to $\mc{V}$ along the slow unstable leaf, denoted by $p_{\mc{V}}$, commutes with $g_t$ we have that 
\[
p_{\mc{V}}\circ h^t_{vw}|_{\mc{H}} = p_{\mc{V}}\circ g_{-t}\circ \mc{I}_{vw} \circ g_t |_{\mc{H}}=g_{-t}\circ p_{\mc{V}} \circ \mc{I}_{vw} \circ g_t |_{\mc{H}}
\]
converges uniformly to $f$ in a neighborhood of $v$. And since $D(p_{\mc{V}}\circ h^t_{vw}) = H^{t,1}_{vw}$ converges, we have that $Df = H^1_{vw}$.

If $\eta\in\mc{H}$ is forward recurrent then for any slow Lyapunov vector $\xi \in E^u_{<1}(\eta)$ we have that $\chi(\eta,\xi)<1$. By Lemma \ref{fast-partial-derivative}, $Df(\eta) =0$. Since almost every vector in $\mc{H}$ is recurrent, $Df =0$ almost everywhere. Moreover, $Df$ is locally bounded, also by Lemma \ref{fast-partial-derivative}. Thus $Df$ equals the zero map in the sense of distributions, and similarly the same holds for all of its higher derivatives. 
Hence, by the Sobolev embedding theorem, $f$ is smooth and $Df=0$ everywhere. It follows that $f$ is constant, i.e. the leaf $\mc{H}$ locally maps to one leaf.

By connectedness of the leaves, $h_{v,w}$ preserves the entire leaf. Now consider the collection of slow unstable leaves with almost every vector being recurrent. By the same argument, such leaves map to slow unstable leaves. Since the foliation by slow unstable leaves is $C^1$ in $W^u(v)$, such leaves are generic by Fubini. Hence every leaf maps to a leaf by continuity.
\end{proof}

\begin{corollary}
The slow unstable distributions are  trivial.
\end{corollary}

\begin{proof}
For any $v,w\in\mc{Q}$ with $w\not\in \cup_t g^tW^u(v)$, the projections of the slow unstable foliations on $W^u(v)$ and $W^u(w)$ to $\partial \til{M}$ agree off of the backward endpoints of the geodesics through $v$ and $w$ in $\partial \til{M}$ by Theorem \ref{lem:holonomy2}.   

Hence we obtain a common $C^0$ foliation of $\partial \til{M}$. Moreover, this foliation is invariant under $\pi_1 M$ since the $E_{<1}^{u}$ distributions are $\pi_1(M)$ invariant. Again by Foulon \cite[{{\em  Corollaire}\rm }]{Foulon1994} this foliation is trivial.
\end{proof}

   The last step in the proof of our main result is now essentially the same as in the strict pinching case.  
\vspace{.6em}

\begin{proof}[Proof of Theorem \ref{thm:main}] (non-strict $\frac14$-pinching case):  Since the slow unstable distribution is trivial, all unstable Jacobi fields belong to $E^u_1$.  Hence  all sectional curvatures are $-1$ on $\mc{Q}$.  Since $\mc{Q}$ is dense in $SM$,  it follows that all sectional curvatures are $-1$.   \end{proof}

\begin{proof}[Proof of Corollary \ref{cor:measures}]
Since $\mu$ is ergodic and invariant, the set of vectors that are recurrent with positive frequency has full measure. Since $\mu$ has full support this set is therefore dense. By Lemma \ref{maxLyap}, the unstable Lyapunov space of exponent $1$ for the geodesic through $v$ coincides with $\mc{E}(v)$ everywhere on this set. In particular $\mc{E}(v)$ has positive dimension everywhere, and the hyperbolic rank of $M$ is positive. 
\end{proof}

\providecommand{\bysame}{\leavevmode\hbox to3em{\hrulefill}\thinspace}
\providecommand{\MR}{\relax\ifhmode\unskip\space\fi MR }
\providecommand{\MRhref}[2]{%
	\href{http://www.ams.org/mathscinet-getitem?mr=#1}{#2}
}
\providecommand{\href}[2]{#2}

\end{document}